\documentclass[twoside]{amsart}
\usepackage{latexsym}
\usepackage{amssymb}
\usepackage{amsmath}
\usepackage{amscd}
\usepackage{mathrsfs}
\usepackage{color}
\usepackage{url}

\usepackage{graphicx}
\newtheorem{theorem}{Theorem}[section]
\newtheorem{lemma}[theorem]{Lemma}
\newtheorem{corollary}[theorem]{Corollary}
\newtheorem{proposition}[theorem]{Proposition}

\theoremstyle{definition}

\newtheorem{definition}[theorem]{Definition}
\newtheorem{example}[theorem]{Example}

\theoremstyle{remark}

\newtheorem{remark}[theorem]{Remark}

\def\tdlc{t.d.l.c.}

\DeclareMathOperator{\Aut}{{\rm Aut}}

\DeclareMathOperator{\stab}{{\rm Stab}}
\DeclareMathOperator{\Fix}{{\rm Fix}}

\def\ident{1}

\def\NN{\mathbb N}
\def\ZZ{\mathbb Z}
\def\tree{\mathcal{T}}

\newcommand{\Pk}[1]{IP_{{#1}}}

\newcommand{\colvec}[1]{\begin{pmatrix}#1\end{pmatrix}}

\def\lcm{\operatorname{lcm}}

\newcommand{\AbertGlasner}{MR2491892}

\newcommand{\Serre}{MR1954121}
\newcommand{\HypFlatRank}{MR2956246}

\newcommand{\FlatRankBuildings}{MR2356316}

\newcommand{\WillStructure}{MR1299067}

\newcommand{\vDant}{MR1556954}
\newcommand{\WillisFlatRank}{MR2052362}
\newcommand{\BurgerMozes}{MR1839488}

\newcommand{\ElderWillis}{ElderWillis}
\newcommand{\WillSimple}{MR2320465}
\newcommand{\ChrisPhD}{ChrisPhD}
\newcommand{\Sunic}{MR2823790}
\newcommand{\Tits}{MR0299534}
\newcommand{\Kap}{MR1703086}
\newcommand{\HaglundPaulin}{MR1668359}
\newcommand{\GardinerPraeger}{MR1279076}
\newcommand{\CarboneGarland}{MR2017720}

\newcommand{\CapraceMonod}{MR2739075}
\newcommand{\CapraceRemy}{MR2217912}
\newcommand{\ConderLorimer}{MR1007714}
\newcommand{\DjokovicMiller}{MR586434}
\newcommand{\MollerVonk}{MR2997026}

\newcommand{\Wilson}{MR1691054}
\newcommand{\Djokovic}{MR0349486}

\begin{document}

\title[Simple groups of automorphisms of trees]{Simple groups of automorphisms of trees determined by their actions on finite subtrees}

\author[C. Banks]{Christopher Banks}
\thanks{The first author is supported by an Australian Postgraduate Award (APA). Research supported by the Australian Research Council projects DP0984342 and DP120100996}
\address{University of Newcastle, 2308, AUSTRALIA}
\email{christopher.banks@newcastle.edu.au}

\author[M. Elder]{Murray Elder}
\address{University of Newcastle, 2308, AUSTRALIA}
\email{murray.elder@newcastle.edu.au}

\author[G. A. Willis]{George A. Willis}
\address{University of Newcastle, 2308, AUSTRALIA}
\email{george.willis@newcastle.edu.au}

\date{\today}


 
\begin{abstract}
We introduce the notion of the $k$-closure of a group of automorphisms of a locally finite tree, and give several examples of the construction. We show that the $k$-closure satisfies a new property of automorphism groups of trees that generalises Tits' Property $P$. We prove that, apart from some degenerate cases, any non-discrete group acting on a tree with this property contains an abstractly simple subgroup.
\end{abstract}

\maketitle

\section{Introduction}
\label{sec:intro}

Simple groups and their classification are a vital part of the structure theory of the classes of groups that admit composition series, among which are the finite groups and the Lie groups. Although totally disconnected, locally compact (\tdlc) groups do not admit composition series, it was shown in \cite{\CapraceMonod} that decomposing such groups into simple pieces plays a role in their structure theory as well. In \cite{\WillSimple} the third author showed that the local and global structures of simple \tdlc\ groups are linked when the group is compactly generated, and that invariants of the group, such as the scale \cite{\WillStructure} and flat-rank \cite{\WillisFlatRank}, could possibly be parameters used in a classification of such groups. 

In this article we present a general construction that produces many new examples of simple compactly generated \tdlc\ groups acting on trees. Particular classes of simple compactly generated \tdlc\ groups have been studied in various contexts. These include:
\begin{itemize}
\item Lie groups over fields of $p$-adic numbers and over fields of formal Laurent series over some finite residue field, where the scale is a power of the characteristic of the residue field and the flat-rank equals the usual algebraic rank;
\item completions of Kac-Moody groups over finite fields \cite{\CarboneGarland}\cite{\CapraceRemy}, where again the scale is a power of the characteristic of the residue field and the flat-rank equals the algebraic rank \cite{\FlatRankBuildings}; 
\item automorphism groups of buildings with negative curvature \cite{\HaglundPaulin}, where the flat-rank is at most~1 \cite{\HypFlatRank};
\item groups of almost automorphisms of trees \cite{\Kap}, where the flat-rank is infinite and the scale depends on valencies of the tree \cite{WiNeretin}; and
\item certain closed subgroups of automorphism groups of trees studied by Burger and Mozes \cite{\BurgerMozes} and Tits \cite{\Tits}, where the flat-rank is at most~1 and the scale depends on valencies of the tree. 
\end{itemize} 

Currently it is not known how close this list is to being exhaustive, or if a classification is possible even for those cases where the flat-rank is 1. Along with those that do not act on trees, there are several distinct subclasses of groups acting on trees, including the rank~1 Lie and Kac-Moody groups. 

The groups constructed and studied here extend and are motivated by the last class of examples. The constructions in \cite{\BurgerMozes} and \cite{\Tits} satisfy Jacques Tits' \textit{Property P}. Tits showed that if a group acting on a tree has this property, then a certain closed subgroup must be simple (apart from some obvious degenerate cases). Amann \cite{Amann} defines a slightly weaker {\em Independence	 Property} which coincides with Property P for closed subgroups of the full automorphism group of the 
tree.

In this paper we define a family of independence properties called \textit{Property $\Pk{k}$} which generalise the Independence Property of Amann. We show that on closed subgroups of the full automorphism group of the tree, this family coincides with another family of properties called \textit{Property $P_k$}, which generalise Tits' Property $P$. Property $P_k$ is used in proving Theorem \ref{thm:simple}, which is an analogue of Tits' theorem, and states that groups with one of these properties (aside from the same degenerate cases) also contain a simple subgroup.

We also provide a general method for constructing groups with these properties from any group acting on a tree. Given any natural number $k$ and a group $G$ acting on a tree $\tree$, then $G^{(k)}$ is defined to be the set of all automorphisms that, on each ball of radius $k$ in $\tree$, agree with some element of $G$. This forms a closed subgroup of $\Aut(\tree)$ called the \emph{$k$-closure} of $G$, which satisfies Property $\Pk{k}$.  

The article is organised as follows. 
In Section \ref{sec:notation} we give the relevant terminology on automorphism groups of trees and graphs, and recall several results from Tits' paper.
In Section \ref{sec:k-closure} we define the $k$-closure of a group acting on a tree, giving a simple example and proving basic facts about the construction, in particular conditions under which the resulting groups are non-discrete. 
In Section~\ref{sec:examples} we apply the $k$-closure construction to some known examples of groups acting on trees. 
In Section \ref{sec:PropertyPk} we define the Independence Property $\Pk{k}$ and show that the $k$-closure of a group satisfies $\Pk{k}$. We also show that this property characterises precisely when the sequence of $k$-closures terminates at $G^{(k)} = \overline{G}$. 
In Section \ref{sec:propertytits} we define Property $P_k$ and establish the relationship between this property and Property $\Pk{k}$.
In Section~\ref{sec:simplicity} we prove the simplicity result (Theorem \ref{thm:simple}) for groups satisfying Property $P_k$. By this theorem, we now have a general method for finding simple groups acting on trees, which we discuss in Section \ref{sec:other}.  We prove some results about the simple groups obtained in this way, including the existence of infinite families of distinct closed simple groups acting on a tree that do not have property $P$.
 
We note that in recent work  M\"oller and Vonk \cite{\MollerVonk} also define a new property of groups acting on trees they call {\em Property H}, which is strictly weaker than Property $\Pk{k}$. Groups that have Property H contain a topologically simple subgroup; that is, it contains no non-trivial closed normal subgroups.

\section{Preliminaries}\label{sec:notation}
For definitions and terminology concerning graphs, we refer to Serre \cite{\Serre}. A graph $X$ is determined by its vertex set $V(X)$ and its set of directed edges $E(X)$. An edge $e\in E(X)$ can be written in the form $(o(e),t(e))\in V(X)\times V(X)$, and every edge has a \textit{inverse edge} $(t(e),o(e))$, which we denote by $\overline{e}$. The term \textit{edge-pair} refers explicitly to the pair $\{e,\overline{e}\}$. In this paper all graphs are assumed to be simple (that is, they have no loops or multiple edges) and locally finite.

Let $\Aut(X)$ denote the group of automorphisms of $X$. If $x,y\in \Aut(X)$ and $v\in V(X)$ then $x.v\in V(X)$ is the image of $v$ under $x$. We may write multiplication in $\Aut(X)$ with or without the composition symbol; in both cases it is always performed from right to left i.e. $(yx).v=y.(x.v)$. An \textit{edge-inversion} is an automorphism $g$ of $X$ satisfying $g(e)=\overline{e}$ for some $e\in E(X)$.

In this paper $\tree$ will denote a tree, and the regular (or homogenous) tree of degree $d$ (in which every vertex is adjacent to $d$ others) is denoted $\tree_d$. Given an edge $(v,w)\in E(\tree)$, the \textit{semi-tree} $\tree_{(v,w)}$ is defined as the connected component of $\tree\backslash\{(v,w),(w,v)\}$ containing $w$. Let $B(v,k)$ be the subtree formed by the closed ball (with respect to the standard metric on $\tree$) of radius $k$ centered at the vertex $v\in V(\tree)$.

If $G\leq\Aut(X)$ and $Y$ is a proper subgraph of $X$ then the set of all $g\in G$ that \textit{stabilise} $Y$ (that is, for which $g(Y)=Y$) is denoted by $\stab_G(Y)$, and the set of $g\in G$ that \textit{fix} $Y$ (that is, $g.v=v$ for all $v\in V(Y)$) is denoted by $\Fix_G(Y)$).

We will say that a group $G$ of automorphisms of $X$ is \textit{vertex-transitive} if for some vertex $v$ the orbit $G.v=V(X)$, and \textit{edge-transitive} if for some edge pair the orbits $G.e\cup G.\overline{e}=E(X)$. For any vertex $v$ let $E(v):=\{e\in E(X):o(e)=v\}$ denote the set of edges emanating from $v$; note that $E(v)$ is stabilised by $\Fix_G(v)$. We define the \textit{local action} of $G$ at $v$ to be the permutation group induced by the action of $\Fix_G(v)$ on $E(v)$, and we say $G$ is \textit{locally transitive} if the local action at every vertex is transitive. 

Recall that an \textit{infinite path} in $\tree$ is represented by a sequence $C=(e_1,e_2,...)$ of edges where $t(e_i)=o(e_{i+1})$ for all $i$. A sequence which is infinite in both directions represents a \textit{doubly-infinite path} $(...,e_{-1},e_1,e_2,...)$, where $t(e_i)=o(e_{i+1})$ for all $i\in\NN$. In this paper all paths are assumed to have no backtracking, that is, $o(e_{i})\ne t(e_{i+1})$ for all $i\in\ZZ$. If the graph is a tree this implies paths have no self-intersection at all. If $C,C'$ are two paths then their intersection is either empty or a path itself.
The \emph{boundary}, $\partial\tree$, of $\tree$ is the set of equivalence classes, $[C]$, of infinite paths in $\tree$, where two infinite paths, $C$ and $C'$, are equivalent if and only if $C\bigcap C'$ is an infinite path. Elements of $\partial\tree$ are known as the \textit{ends} of $\tree$, and a path $C$ is called a \textit{representative of the end} $b$ if $b=[C]$. We say that an automorphism $g\in G\leq\Aut(\tree)$ \textit{stabilises the end} $b$ if the image under $g$ of any representative of $b$ is another representative of $b$, and that $g$ \textit{fixes the end} $b$ if it fixes some representative of $b$. 

The automorphism group of a tree can be equipped with a topology, whose basis is given by the collection of all sets of the form
\[\mathscr{U}(x,\mathcal{F}):=\{y\in\Aut(\tree):y.v=x.v\;\forall v\in\mathcal{F}\}\]
where $x\in \Aut(\tree)$ and $\mathcal{F}$ is a finite vertex set. Under this topology $\Aut(\tree)$ is a topological group. The open set $\Fix(v)=\mathscr{U}(1_G,\{v\})$ is a profinite group, as it can be expressed as the projective limit of the finite groups $\Aut(B(v,k))$. Since profinite groups are compact and totally disconnected \cite{\Wilson}, it follows that $\Aut(\tree)$ is a \tdlc\ group, with compact open subgroups $\Fix(\mathcal{F})$ for all finite $\mathcal{F}$. 

Recall that every neighbourhood of the identity $1_G$ in a \tdlc\ group $G$ contains a compact open subgroup $U$ \cite{\vDant}.
The fact that a compact group is discrete if and only if it is finite implies the following lemma.
\begin{lemma}\label{lem:finite-discrete} A subgroup $G\leq\Aut(\tree)$ is discrete (with the subspace topology induced by the one given above) if and only if there exists a vertex for which $\stab_G(v)(=\Fix(v))$ is a finite group.
\end{lemma}

\subsection{Results of Jacques Tits}

The property we will study in Section \ref{sec:PropertyPk} is based on a property defined by Jacques Tits, which he used to find simple groups.
\begin{definition}\cite[Section 4]{\Tits} Suppose $G\leq\Aut(\tree)$, $C$ a path in $T$ and $\Fix_G(C)$ is the fixator of the path $C$. Define $\pi$ to be the projection of $V(\tree)$ onto $V(C)$ where $\pi(v)=x$ if $x$ is the closest vertex in the path $C$ to $v$. Let $\Fix_G(C)_x$ denote the action of $\Fix_G(C)$ restricted to $\pi^{-1}(x)$. Then $G$ satisfies \textit{Property $P$} if for any such $C$ we have
$\Fix_G(C)=\prod_{x\in V(C)}\Fix_G(C)_x$.
\label{def:Tits}\end{definition}

\begin{remark}\label{remark:indep_prop} In the case where $C$ is just an edge $(v,w)\in E(\tree)$, Property $P$ states that $\Fix_G((v,w))$ decomposes into two independent actions on $\tree_{(v,w)}$ and $\tree_{(w,v)}$, and consequently $\Fix_G((v,w))=\Fix_{G}(\tree_{(v,w)})\times\Fix_{G}(\tree_{(w,v)})$. This statement for edges defines the weaker Independence Property \cite[Definition 9]{Amann} which is shown to be equivalent to Tits' Property $P$ if $G$ is closed.\end{remark}

It is also mentioned in \cite{\Tits} that groups that stabilise a proper subtree of $\tree$ or an end of $\tree$, contain many normal subgroups. For instance, if $G=\Fix(v)$ is the stabiliser of a vertex in the full automorphism group of $\tree$, then the groups $\Fix(B(v,r))$ are normal subgroups of $G$ for each $r\in\NN$.
We will make use of the following three results, which are relevant to these cases.
 
 \begin{lemma}\cite[Lemma 4.1]{\Tits}\label{lem:T4.1} Let $G$ be a group of automorphisms of a tree $\tree$. Then the following are equivalent:
 \begin{enumerate}
 \item $G$ does not stabilise a proper non-empty subtree of $\tree$;
 \item The orbit $G.v$ of any vertex $v\in V(\tree)$ has non-empty intersection with any semi-tree.
 \end{enumerate}
 \end{lemma}
 
\begin{proposition}\cite[Proposition 3.4]{\Tits}\label{prop:T3.4} If $G\leq\Aut(\tree)$ contains no translations then $G$ is contained in either the stabiliser of a vertex, the stabiliser of an edge or the fixator of an end of $\tree$.
\end{proposition}

\begin{lemma}[\cite{\Tits} Lemme 4.4] Suppose $\tree$ is a tree that is not a doubly-infinite path and $G,H$ non-trivial subgroups of $\Aut(\tree)$ such that $G$ normalises $H$. If $G$ does not stabilise a proper non-empty subtree or an end of $\tree$, then $H$ also does not stabilise a proper non-empty subtree or an end of $\tree$.
\label{lem:normalised}\end{lemma}

\section{The $k$-closure of a group of automorphisms}
\label{sec:k-closure}

Let $x$ be an automorphism of $\tree$, $v$ a vertex and $k$ a natural number. 
Since graph automorphisms preserve distance,  $x$ 
maps $B(v,k)$ to $B(x.v,k)$. Let $x|_{B(v,k)}$ denote the map  $B(v,k)\rightarrow B(x.v,k)$ defined by $x|_{B(v,k)}.w=x.w$, which we call the  \emph{restriction of $x$ to $B(v,k)$}.
If $U\subseteq \Aut(\tree)$ let $U|_{B(v,k)}=\{x|_{B(v,k)} \mid x\in U\}$.
If $y\in\Aut(\tree)$ then \[
(yx)|_{B(v,k)}=y|_{B(x.v,k)}\circ x|_{B(v,k)}
\]
In particular, $x^{-1}|_{B(x.v,k)}\circ x|_{B(v,k)}$ and  $ x|_{B(x^{-1}.v,k)}\circ x^{-1}|_{B(v,k)}$ act trivially on $B(v,k)$.

\begin{definition}
\label{defn:k-closure}  
For $G\leq \Aut(\tree)$ and $k\in \NN$, define the \emph{$k$-closure of $G$} to be 
$$ 
G^{(k)} = \left\{ x\in \Aut(\tree) \mid  \forall\;v\in V(\tree)\;\exists\;g\in G \hbox{ such that } g|_{B(v,k)} = x|_{B(v,k)} \right\}.
$$
\end{definition}

\begin{lemma}$G^{(k)}$ is a subgroup of $\Aut(\tree)$.
\begin{proof}
If $x,y\in G^{(k)}$ and  $v$ is any vertex, then we have  $g_1,g_2\in G$ with $g_1|_{B(v,k)} = x|_{B(v,k)}$,  and $g_2|_{B(x.v,k)} = y|_{B(x.v,k)}$.
Then 
 \[
(y\circ x)|_{B(v,k)} = y|_{B(x.v,k)}\circ x|_{B(v,k)}= g_2|_{B(x.v,k)}\circ g_1|_{B(v,k)} = (g_2\circ g_1)|_{B(v,k)}\] so $y\circ x\in G^{(k)}$.
In addition we have $g\in G$ with $g|_{B(x^{-1}.v,k)} = x|_{B(x^{-1}.v,k)}$, so $x^{-1}|_{B(v,k)} = g^{-1}|_{B(v,k)}$ and hence $x^{-1}\in G^{(k)}$.
\end{proof}\end{lemma}
 
The $k$-closure of $G$ consists of automorphisms of $\tree$ that on each ball of radius $k$ in $\tree$ agree with some element of $G$. It is clear from this that the local actions of $G$ and $G^{(k)}$ are identical. 

The role of the group $G$ is to provide a list of ``allowed'' actions for each ball. In this sense the construction is comparable to the \textit{universal groups} defined in \cite[\S3.2]{\BurgerMozes}. They consist of automorphisms of $\tree$ that on each ball of radius 1 perform an ``allowed'' permutation from some permutation group $F$, which is isomorphic to the local action of the group. This idea is illustrated by the following example. 

\begin{example}\label{eg:egfirst}
Consider the following subgroup $G$ of the automorphism group of the ternary tree $\tree_3$. Let $i : E(\tree_3) \to \{1,2,3\}$ be an edge-labeling where $i(\overline{e}) = i(e)$ for each edge $e$, and every vertex is incident on one edge of each label (see \cite[\S3.2]{\BurgerMozes}). Then for each $v\in V(\tree_3)$  the restriction $i|_{E(v)}$ of $i$ to $E(v)$ is a bijection on $\{1,2,3\}$, and it follows that for any automorphism $x\in\Aut(\tree_3)$ and any $v\in V(\tree_3)$ the map 
\[\pi_{x,v}=i|_{E(x.v)} \circ x \circ \left(i|_{E(v)}\right)^{-1}\] 
is a permutation of $\{1,2,3\}$. 
Let
$$
G = \left\{ x\in\Aut(\tree_3) \mid \pi_{x,v}=\pi_{x,w}\forall v,w\in V(\tree_3)\right\},
$$
be the group of automorphisms that act as the {\em same} permutation around each vertex. This is a subgroup of $\Aut(\tree_3)$ since
\begin{equation}
\label{eq:multiply}
i|_{E((yx).v)} \circ (yx) \circ \left(i|_{E(v)}\right)^{-1} =i|_{E((yx).v)} \circ y\circ\left(i|_{E(x.v)}\right)^{-1} \circ i|_{E(x.v)}\circ x \circ \left(i|_{E(v)}\right)^{-1}.
\end{equation}
Let $S_3$ denote the group of permutations of $\{1,2,3\}$, and define $\pi : G\to S_3$  by $\pi(g) = \pi_{g,v}$. This is well defined (since $\pi_{g,v}$ is the same for any $v$) and a surjective homomorphism by Equation (\ref{eq:multiply}). From this it follows that for any $v,w\in V(\tree_3)$ and for any $\sigma\in S_3$ there exists exactly one $g\in G$ such that $g.v=w$ and $\pi(g)=\sigma$. The special case when $w=v$ implies that $\stab_G(v)\cong S_3$ for all vertices $v$, and hence by Lemma \ref{lem:finite-discrete} $G$ is discrete.

Recall any automorphism $x$ of $\tree_3$ is assigned permutations $\pi_{x,v}$ for all vertices $v$. From above there always exists $g_v\in G$ with $\pi(g_v)=\pi_{x,v}$ and where $g_v$ maps $v$ to $x.v$. Now $x$ does the same permutation as $g_v$ at $v$, and hence they agree on the ball $B(v,1)$. Therefore $x\in G^{(1)}$ and so the $1$-closure of $G$ is the full automorphism group $\Aut(\tree_3)$. 

On the other hand, if an automorphism is in the $2$-closure, then it must be the same permutation around a vertex $u$ and an adjacent vertex $v$, and also the same permutation around $v$ and a third vertex $w$ next to it, which means it is the same permutation around every vertex. Hence the $2$-closure $G^{(2)}$ is equal to $G$. By a similar argument all $k$-closures are equal to $G$ for $k\geq 2$.

\end{example}

In Section \ref{sec:examples} we will see more interesting examples of groups arising from the $k$-closure construction. 
The remainder of this section records several facts about $k$-closures. The first results explain the sense in which these groups are a `closure' of $G$. 

\begin{proposition}\label{prop:basic}
Let $G\leq\Aut(\tree)$ and $k\in\NN$.
\begin{enumerate}
\item $G^{(k)}$ is a closed subgroup of $\Aut(\tree)$.
\label{prop:basic1}
\item $G^{(r)}\leq G^{(k)}$ for all $r>k$.
\label{prop:basic2}
\item $\bigcap_{k\in \NN} G^{(k)} = \overline{G}$, the closure of $G$.
\label{prop:basic3}
\item $G^{(l)} = (G^{(k)})^{(l)}$ whenever $l\leq k$.  
\label{prop:basic_idempotent}
\item The orbit $G^{(k)}.v$ is equal to $G.v$ for every $v\in V(\tree)$.
\label{prop:basic4}
\end{enumerate}
\end{proposition}

\begin{proof} (i) Since 
\begin{eqnarray*}
\Aut(\tree)\backslash G^{(k)}&=&\{x\in\Aut(\tree):\exists v_x\in V(\tree)\mbox{ with }x|_{B(v_x,k)} \neq g|_{B(v_x,k)}\forall g\in G\}\\
&=&\bigcup_{x\notin G^{(k)}}\mathscr{U}(x,B(v_x,k))
\end{eqnarray*}
is the union of open sets, $G^{(k)}$ is closed.

(ii)  If $x\in G^{(r)}$ then for every vertex $v$ there is some $g\in G$ with  $g|_{B(v,r)} = x|_{B(v,r)}$, and since $r>k$ we have  $g|_{B(v,k)} = x|_{B(v,k)}$.

(iii) Since $g\in G$ agrees with itself everywhere,  $G$ is contained in $G^{(k)}$ for every~$k$. Thus $\bigcap_{k\in \NN} G^{(k)}$  contains $G$, and is closed by~(\ref{prop:basic1}), and so  $\bigcap_{k\in \NN} G^{(k)}\supseteq\overline{G}$. For the reverse inclusion it is enough to show that any open set containing $x\in\bigcap_{k\in \NN} G^{(k)}$ also contains some $g\in G$, which holds because the sets $\mathscr{U}(x,B(v,k))$ ($v\in V(\tree)$, $k\in\NN$) form a basis for the subspace topology on $\bigcap_{k\in \NN} G^{(k)}$ and each contains at least one $g\in G$ such that $g|_{B(v,k)} = x|_{B(v,k)}$. 

(iv) Since $G\leq G^{(k)}$ then $G^{(l)}\leq(G^{(k)})^{(l)}$. Suppose $x\in(G^{(k)})^{(l)}$, then for all $v\in V(\tree)$ there exist $y_v\in G^{(k)}$ such that $y_v|_{B(v,l)} = x|_{B(v,l)}$. Note that since $l\leq k$ we have $B(v,l)\supseteq B(v,k)$. Since $y_v\in G^{(k)}$ there exists $g_v\in G$ such that $y_v|_{B(v,l)}=g|_{B(v,l)}$. Then $x|_{B(v,l)}=g|_{B(v,l)}$ which implies that $x\in G^{(l)}$.

(v) Since every $x.v\in G^{(k)}.v$ satisfies $x.v=g.v$ for some $g\in G$, we know that $G^{(k)}.v$ is contained in $G.v$; equality follows because $G\leq G^{(k)}$.
\end{proof}

The question of what happens to $(G^{(k)})^{(l)}$ when $l>k$ is more subtle, and will be discussed in Section \ref{sec:PropertyPk}.

The next result gives a criterion for when the $k$-closure of a group is non-discrete. This will prove to be necessary when attempting to construct new examples of simple groups.
\begin{theorem}
\label{thm:non-discrete}
Let $G\leq \Aut(\tree)$, fix $k\in \NN$ and suppose that $G$ does not stabilise any proper subtree of $\tree$. Then $G^{(k)}$ is non-discrete if and only if there is $(v,w)\in E(\tree)$ and $g\in G$ such that
\begin{equation}
\label{eq:non-discrete}
g|_{B(v,k)\cap B(w,k)} = \ident\ \hbox{ and } \ g|_{B(w,k)}\ne \ident.
\end{equation}
Equivalently, $G^{(k)}$ is discrete if and only if $\Fix_G(B(v,k)\cap B(w,k))=\{\ident\}$ for every $(v,w)\in E(\tree)$.
\end{theorem}
\begin{proof} 
If $G^{(k)}$ is non-discrete, then for any vertex $u$ there is a non-identity element $h\in \stab_{G^{(k)}}(u)$ such that $h|_{B(u,k)} = \ident$. Since $\tree\setminus\{u\} = \bigcup_{t\in B(u,1)} \tree_{(u,t)}$, there exists a vertex $t$ adjacent to $u$ such that $h|_{\tree_{(u,t)}}\ne \ident$. There are two cases; in the first case $h|_{B(u,k)}=\ident$ and $h|_{B(t,k)}\neq 1$, so set $(v,w)=(u,t)$. Otherwise there must be an edge $(v,w)$ in $\tree_{(u,t)}$ such that $h|_{B(v,k)} = \ident$ and $h|_{B(w,k)} \ne \ident$. Choosing $g\in G$ such that $h|_{B(w,k)} = g|_{B(w,k)}$ shows that (\ref{eq:non-discrete}) holds.

For the converse, assume that (\ref{eq:non-discrete}) holds for some $(v,w)\in E(\tree)$ and let $u\in V(\tree)$ and $m\in \NN$. It will be shown that there is a non-identity element $h\in G^{(k)}$ with $h|_{B(u,m)} = \ident$. Note first of all that, if $t$ is any vertex in $\tree_{(v,w)}$, then $B(v,k)\cap B(t,k) \subseteq B(v,k)\cap B(w,k)$ and, if $t$ is any vertex in $\tree_{(w,v)}$, then $B(w,k)\cap B(t,k) \subseteq B(w,k)\cap B(v,k)$. Since $g|_{B(v,k)\cap B(w,k)} = \ident$, the element $h_1$ defined by
$$
h_1|_{\tree_{(v,w)}} = g|_{\tree_{(v,w)}} \hbox{ and }
h_1|_{\tree_{(w,v)}} = \ident|_{\tree_{(w,v)}}
$$
is contained in $G^{(k)}$. Since no proper sub-tree of $\tree$ is invariant under $G$, Lemma \ref{lem:T4.1} implies that the orbit $G.u$ intersects every semi-tree of $\tree$. Choose a semi-tree $\tree_{(r,s)}$ that is contained in $\tree_{(w,v)}$ and such that $d(v,r) \geq m$ and choose $x\in G$ such that $x.u\in \tree_{(r,s)}$. Then $B(x.u,m)\cap \tree_{(v,w)} = \emptyset$ and it follows that $h = xh_1x^{-1}$ belongs to $G^{(k)}$, is non-trivial and $h|_{B(u,m)} = \ident$.

For the final claim, note that, if $G^{(k)}$ is not discrete, then $\Fix_G(B(v,k)\cap B(w,k))$ is not trivial for the edge $(v,w)$ in~\eqref{eq:non-discrete}. On the other hand, if $G^{(k)}$ is discrete, then $g|_{B(v,k)\cap B(w,k)} = \ident$ implies that $g|_{B(w,k)} = \ident$ for every $g\in G$ and $(v,w)\in E(\tree)$. Hence, if $g|_{B(v,k)\cap B(w,k)} = \ident$, then $g|_{B(v',k)\cap B(w,k)} = \ident$ for every $v'$ adjacent to $w$, which in turn implies that $g|_{B(v',k)} = \ident$ for every $v'$ adjacent to $w$, whence $g|_{B(v',k+1)} = \ident$. Continuing by induction shows that $g|_{B(v',j)} = \ident$ for every $j>k$, that is, that $g = \ident$. 
\end{proof}

\begin{figure}[ht]\begin{center}
\includegraphics[height=5cm]{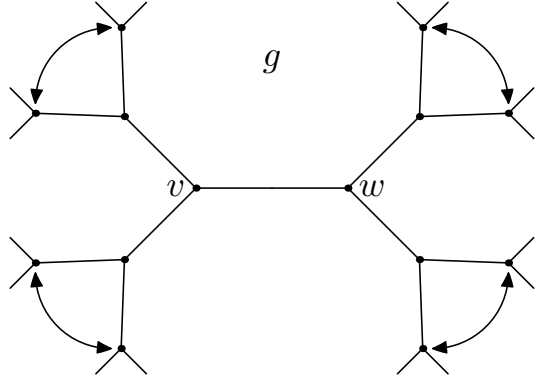}
\caption{An automorphism $f\in G$ that satisfies Equation \eqref{eq:non-discrete} in Theorem \ref{thm:non-discrete} for $k=2$. In the proof the corresponding $h_1\in G^{(k)}$ is defined to fix the left semi-tree and agree with $f$ on the right semi-tree.}
\label{fig:non-discrete}
\end{center}\end{figure}

\begin{corollary}\label{cor:nondiscrete} Suppose $G\leq \Aut(\tree)$ acts with finitely many orbits on $\tree$ and does not stabilise any proper non-empty subtree. Then $G$ is non-discrete if and only if $G^{(k)}$ is non-discrete for infinitely many (and hence all) $k\in\NN$.
\begin{proof} Suppose that $G^{(k)}$ is non-discrete for infinitely many $k$. Since the action of~$G$ on $\tree$ is co-compact, there are only finitely many $G$-orbits in $E(\tree)$. By the pigeonhole principle it may be assumed when applying Theorem~\ref{thm:non-discrete} that the edge $(v,w)$ is always the same. Then the theorem gives an infinite number of elements $g\in G$ that fix $w$. On the other hand, if there is any $j\in\NN$ for which $G^{(j)}$ is discrete then every subgroup, in particular $G$ (and each $G^{(k)},k>j$), is discrete.
\end{proof}\end{corollary}

\begin{corollary}\label{cor:discrete} Suppose $G\leq\Aut(\tree)$ does not stabilise any proper subtree of $\tree$ and also suppose that $G^{(k)}$ is discrete for some $k\in\NN$. Then $G^{(k)}=G$.
\end{corollary}
\begin{proof} By Proposition \ref{prop:basic}(\ref{prop:basic3}) we have that $G\leq G^{(k)}$. Now suppose $x$ is in $G^{(k)}$ and let $(v,w)\in E(\tree)$. By definition of $G^{(k)}$, there exist $g_v,g_w\in G$ such that $g_v|_{B(v,k)} = x|_{B(v,k)}$ and $g_w|_{B(w,k)} = x|_{B(w,k)}$. Then $g_v$ and $g_w$ have the same action on $B(v,k)\cap B(w,k)$  and hence $g_v^{-1}g_w$ fixes $B(v,k)\cap B(w,k)$. Since $G^{(k)}$ is discrete Theorem~\ref{thm:non-discrete} implies that $g_v^{-1}g_w$ is the identity automorphism; that is, $g_v=g_w$. Applying this to every edge, it follows that $x$ agrees with the same element of $G$ on every ball of radius $k$, and hence $x\in G$.
\end{proof}

The last results in this section gives general criteria for two subgroups of $\Aut(\tree)$ to produce the same $k$-closures. 

\begin{proposition}\label{prop:kequal} Two groups $G,H\leq\Aut(\tree)$ satisfy $G^{(k)}=H^{(k)}$ if and only if for each $g\in G$ and $v\in V(\tree)$ there is $h_v\in H$ such that $(h_v^{-1}g)|_{B(v,k)}$ belongs to $\stab_H(v)|_{B(v,k)}$, and similarly for $h\in H$.
\end{proposition}

\begin{proof} 
Assume that $G^{(k)}=H^{(k)}$ and consider $g\in G$ and $v\in V(\tree)$. Since $g\in G^{(k)}$ there exists $h_v\in H$ such that $h_v|_{B(v,k)} = g|_{B(v,k)}$. Hence $(h_v^{-1}g)|_{B(v,k)} = \ident$, which belongs to $\stab_G(v)|_{B(v,k)}$. A similar argument applies for each $h\in H$. 

Conversely, assume that for each $g\in G$ and $v\in V(\tree)$ there is $h_v\in H$ such that $(h_v^{-1}g)|_{B(v,k)}$ belongs to $\stab_H(v)|_{B(v,k)}$ and consider $x\in G^{(k)}$. For each $v\in V(\tree)$ there is $g_v\in G$ such that $g_v|_{B(v,k)} = x|_{B(v,k)}$ and then, by assumption, there is $h_v\in H$ such that  $(h_v^{-1}g_v)|_{B(v,k)} \in \stab_H(v)|_{B(v,k)}$. Suppose that $(h_v^{-1}g_v)|_{B(v,k)} = a|_{B(v,k)}$ with $a\in H$. Then $h_va\in H$ and $h_va|_{B(v,k)} = x|_{B(v,k)}$. Hence $x\in H^{(k)}$ and it has been shown that $G^{(k)}\leq H^{(k)}$. That $H^{(k)}\leq G^{(k)}$ may be shown similarly. 
\end{proof}

\begin{corollary}\label{cor:kequal}
Suppose we have two groups $G,H\leq\Aut(\tree)$. If
\begin{enumerate}
\item \label{prop:kequal1}
$\stab_G{v}|_{B(v,k)}=\stab_H{v}|_{B(v,k)}$ for all $v\in V(\tree)$; and 
\item \label{prop:kequal2} 
$G,H$ and $G\cap H$ act on $\tree$ with the same orbits,
\end{enumerate}
then $G^{(k)}=H^{(k)}$.
\end{corollary}\begin{proof}
Consider $g\in G$ and $v\in V(\tree)$. By~(\ref{prop:kequal2}), there is $h_v\in G\cap H$ such that $h_v.v = g.v$, so that $h_v^{-1}g\in G$ and belongs to $\stab_G(v)$.  Then~(\ref{prop:kequal1}) implies that $h_v^{-1}g|_{B(v,k)} \in \stab_H(v)|_{B(v,k)}$ as required by Proposition \ref{prop:kequal}. A similar argument applies for $h\in H$ and hence the criterion for $G^{(k)}$ to equal $H^{(k)}$ is established. 
\end{proof}

\begin{remark}\label{rem:finitely_constrained}

(i) The definition of $G^{(k)}$ constrains its elements to act in the same way as those of $G$ up to distance $k$ from each vertex of $\tree$. It is thus a closely related construction to the `{finitely constrained}' groups of automorphisms of rooted trees studied in~\cite{\Sunic}. 
In the case of $k$-closures, basing the constraints on the group $G$ ensures that there do exist automorphisms of the tree that satisfy them. The problem then becomes to determine how much larger~$G^{(k)}$ is than $G$. 

(ii) The $k$-closure is a special case of a much more general construction. Let $G\leq H$ be groups of permutations of some set $X$. Then $G$ and $H$ act on $\mathscr{P}(X)$, the power set of $X$. Let $\mathscr{C}\subseteq \mathscr{P}(X)$ be invariant under the $G$-action and define the \emph{$\mathscr{C}$-closure of $G$ in $H$} to be
$$
G^{\mathscr{C}} = \left\{ x\in H \mid \text{ for each }C\in \mathscr{C},\exists g\in G \text{ such that }x|_C = g|_C\right\}.
$$
Then $G^{\mathscr{C}}$ is a subgroup of $H$ and leaves ${\mathscr{C}}$ invariant. Returning to the case treated here when $H = \Aut(\tree)$, the set $\mathscr{C}$ could be any isomorphism class of subtrees. These subtrees could be finite or infinite such as, for example, the set of infinite paths in $\tree$ or the set of doubly-infinite paths. The proof of Proposition~\ref{prop:basic}(\ref{prop:basic1}) applies to show that each of these ${\mathscr{C}}$-closures yields a closed subgroup of $\Aut(\tree)$. 
\end{remark}

\section{Examples}\label{sec:examples}

In this section we discuss some examples of groups acting on trees, and apply the $k$-closure construction in each case. These examples illustrate some of the results in the previous section, and cover three general constructions that produce a wide variety of groups acting on trees.

We briefly note that Example \ref{eg:egfirst} is one of a family of discrete groups acting vertex-transitively and locally-transitively on a ternary tree, of which there are exactly seven \cite{\ConderLorimer,\DjokovicMiller}. These seven examples show some interesting behaviour in terms of their $k$-closures for small $k$, however the details are quite technical -- see the first author's thesis \cite{\ChrisPhD}.

\subsection{An infinite series of $k$-closures}
In this subsection we show that the group $G=PSL(2,\mathbb{Q}_p)$ acting on its Bruhat-Tits tree has distinct non-discrete $k$-closures for all $k\in\NN$. This is an example of a matrix group over a local field, which is the general case discussed in \cite[\S II.1]{\Serre}.

We begin by giving the structure of the Bruhat-Tits tree. Let $\langle e_1,e_2\rangle$ denote the span over $\ZZ_p$ of two independent vectors $e_1,e_2\in\mathbb{Q}_p^2$ (these are called \textit{lattices of} $\mathbb{Q}_p^2$). 
One such example is $\mathbf{L}_p:=\left\langle\colvec{1\\0},\colvec{0\\p}\right\rangle$.

Let $\mathbf V$ be the set of all such lattices, and define an equivalence relation $\sim$ on $\mathbf V$ such that
\[\mathbf{L}\sim\mathbf{L}'\Leftrightarrow\mathbf{L'}=\lambda\mathbf{L}\mbox{ for some }\lambda\in\mathbb{Q}_p^*.\]
Then $G$ acts on $\mathbf{V}$ by $M:\langle e_1,e_2\rangle\mapsto \langle Me_1,Me_2\rangle$, which preserves the equivalence relation since $\lambda Me_i=M\lambda e_i$.

Define a graph $X$ with vertex set $\mathbf{V}/\sim$ and with an edge $(\mathbf{L},\mathbf{L}')$ if there exist $e_1,e_2$ such that $\mathbf{L}\sim \langle e_1,e_2\rangle$ and $\mathbf{L'}\sim \langle e_1,p^{\pm 1}e_2\rangle$. The action of $G$ preserves edges since 
$\langle Me_1,M(p^{\pm 1}e_2)\rangle=\langle Me_1,p^{\pm 1}(Me_2)\rangle$.

Set $v:=\left[\left\langle\colvec{1\\0},\colvec{0\\1}\right\rangle\right]$, then clearly $v$ is adjacent to $[\mathbf{L}_p]$. For any $f\in\{0,1,\dots,p-1\}$ we have 
$\left[\left\langle\colvec{1\\0},\colvec{f\\1}\right\rangle\right]=\left[\left\langle\colvec{1\\0},\colvec{0\\1}\right\rangle\right]$
and hence $[\mathbf{L}_f]:=\left[\left\langle\colvec{p\\0},\colvec{f\\1}\right\rangle\right]$ is  adjacent to $v$ for all $f\in\{0,1,\dots,p-1\}$. Indeed the vertex $v$ has valency $p+1$, and it follows from \cite{\Serre} that $X=\tree_{p+1}$.

If $M$ is a matrix fixing $v$, then one can write the basis vectors of $\mathbf{L}_0$ as combinations of the column vectors in $M$, which must be in $(\ZZ_p)^2$. Hence $\Fix_G(v)=PSL(2,\ZZ_p)$. 
Similar calculations (involving writing column vectors of $M$ in terms of the basis vectors of a lattice) show that for any $r\geq 0$ 
\[\Fix_G(B(v,r))=\left\{(a_{ij})\in PSL(2,\mathbb{Z}_p):M\equiv\begin{pmatrix}1&0\\0&1\end{pmatrix}(\bmod\;p^r)\right\}.\]

Fix $k\in\mathbb{N}$ and set $M=\begin{pmatrix}1&p^k\\0&1\end{pmatrix}$. Then $M$ fixes $B(v_0,k)$ but not $B(v_0,k+1)$. For any $f\in\{0,\dots,p-1\}$ let $a=f+a_1p+...+a_kp^k$. Then the vertices $\left\{\left[\left\langle \colvec{p^{k+1}\\0},\colvec{a\\1}\right\rangle\right]\right\}$ lie in the ball centred at the vertex $[\mathbf{L}_f]$ (which is adjacent to $v_0$) of radius $k$. But then
\[M\colvec{a\\1}=\colvec{f+a_1p+...+(a_{k}+1)p^{k}\\1}\]
which implies that $M$ has non-trivial action on $B([\mathbf{L}_f],k)$ for any $f\in\{0,\dots,p-1\}$. The same calculation shows any matrix fixing $B(v,k)$ but not $B(v,k+1)$ has non-trivial action on $B([\mathbf{L}_f],k)$ for almost all $f\in\{0,\dots,p-1\}$ (i.e. all but possibly one such $f$). Hence an automorphism $\alpha$ of $\tree_{p+1}$ that fixes $B(v,k)$, agrees with $M$ on $B([\mathbf{L}_f],k)$ (for a fixed $f\neq0$) and fixes $B([\mathbf{L}_{f'}],k)$ for all $f'\neq f$ does not agree with any element of $G$ on $B(v_0,k+1)$, and hence is not in $G^{(k+1)}$.
Thus $G^{(k+1)}\neq G^{(k)}$. 

\label{sec:non-discrete-i}

\subsection{Baumslag-Solitar groups}\label{sec:BS}
Recall that the graph of groups construction produces a group 
that acts on an associated \textit{Bass-Serre tree} \cite{\Serre}. 
An example of a group arising from this construction is the Baumslag-Solitar group 
\[BS(m,n)=\langle a,t \mid ta^mt^{-1}=a^n\rangle.\]
As discussed in \cite{\ElderWillis}, the vertices of the Bass-Serre tree are given by cosets $w\langle a\rangle$, where $w$ is a freely reduced word over the alphabet 
\[\{t,at,\dots, a^{n-1}t, t^{-1}, at^{-1}, \dots, a^{m-1}t^{-1}\},\] 
and directed edges $(u\langle a\rangle,v\langle a\rangle)$ labelled by $t^{\pm1}$ if $v\langle a\rangle=ua^it^{\pm1}\langle a\rangle$ for some $i$.
The resulting tree $\tree_{BS(m,n)}$ is graph isomorphic to $\tree_{m+n}$, and the group acts vertex-transitively on $\tree_{BS(m,n)}$ by acting on the left of cosets. 

\begin{lemma}\label{lem:BS}
Let $w\langle a\rangle$ be a vertex and define $i$ (resp. $j$) to be the number of $t$'s (resp. $t^{-1}$'s) in $w$. Then the subgroup $\langle a^{m^jn^i}\rangle$ is contained in the fixator of $w\langle a\rangle$ under the action of $BS(m,n)$. Consequently, if $A$ is a finite subtree of $\tree_{BS(m,n)}$ then $\Fix_{\langle a\rangle}(A)$ is nontrivial.\end{lemma}
\begin{proof} The relations defining $BS(m,n)$ imply that $a^{cn}t=ta^{cm}$ and $a^{cm}t^{-1}=t^{-1}a^{cn}$ for $c\in \mathbb Z$. By definition $a^{m^jn^i}$ contains enough $a^m$'s and $a^n$'s to commute past $w$; that is, $a^{m^jn^i}w=wa^{m^in^j}$.

Since $A$ is a finite set of vertices, we can let $I$ (resp. $J$) be the maximum number of $t$'s (resp. $t^{-1}$'s) in any word $w$ where $w\langle a\rangle \in A$. Then $a^{m^Jn^I}$ fixes each  vertex in $A$. 
\end{proof}

Note that some infinite paths in the Bass-Serre tree also have nontrivial fixators. For instance, for each nonnegative integer $i$ the vertex $(atat^{-1})^i\langle a\rangle$ is fixed by $a^n$, so the infinite path spanned by these vertices starting at $\langle a\rangle$ has a nontrivial fixator.

Let $\rho$ denote the homomorphism from $BS(m,n)$ to $\ZZ$ that sends a word $w$ to its $t$-exponent sum. Then $BS(m,n)$ preserves the level sets of $\rho$, which form a partition we denote by $\mathcal{P}$.

\begin{proposition}\label{prop:BS} Let $G=BS(m,n)$. Then
\begin{enumerate}
\item The local action of $G^{(k)}$ at any vertex is isomorphic to $\ZZ/(\lcm(m,n))\ZZ$;
\item $G^{(k)}$ preserves $\mathcal{P}$.
\end{enumerate}
\end{proposition}
\begin{proof}
To prove (i) we consider the local action of $G$ at the vertex $v=\langle a\rangle$. This is sufficient since $G^{(k)}$ and $G$ have the same local actions and $G$ acts vertex-transitively. 
The set $B(v,1)$ contains vertices $w\langle a\rangle$ where $w$ has at most one $t,t^{-1}$ letter. Thus $a^j$ fixes these vertices if and only if $j$ is divisible by both $m$ and $n$, so the fixator of $B(v,1)$ in $G$ is  $\langle a^{\lcm(m,n)}\rangle$. 
Then the local action is equal to the quotient group $\Fix_G(v)/\Fix_G(B(v,1))$ which is isomorphic to $\ZZ/(\lcm(m,n))\ZZ$. 

To prove (ii) it is enough to show that $G^{(1)}$ preserves $\mathcal{P}$, since $G^{(k)}\subseteq G^{(1)}$. Since any two are connected by a finite directed path, where each edge is labelled by $t$ or $t^{-1}$, it is sufficient to show that $G^{(1)}$ preserves the labels of the directed edges. This follows immediately from the definition of $G^{(1)}$ since $G$ preserves labelled directed edges.
\end{proof}

In Section \ref{sec:simpleexamples} we will be restricting our attention to the groups $BS(m,n)$ where $m,n$ are relatively prime, in which case $\lcm(m,n)=mn$ and the local action is isomorphic to $\ZZ/m\ZZ\times\ZZ/n\ZZ$. 

Now we discuss the structure of automorphisms $x\in BS(m,n)^{(1)}$ that fix the vertex labelled by $\langle a\rangle$, assuming that $m,n$ are coprime. For each vertex $v\in V(\tree_{BS(m,n)})$ there exists a word $w\in BS(m,n)$ whose action agrees with $x$ on $B(v,1)$. Since $x$ fixes a vertex then $x.v$ and $v$ lie in the same part of $\mathcal{P}$ for all $v\in V(\tree_{BS(m,n)})$ (see Proposition \ref{prop:BS}). Hence $w$ preserves the $t$-exponent sum of $v$, and hence $\rho(w)=0$. Indeed we can assume that $w\in\langle a\rangle$.

To construct any such automorphism, begin by assigning to the vertex $\langle a\rangle$ an element $a^i$ where $0\leq i\leq mn-1$. Then proceeding inductively, assign a number $\sigma_v$ to each vertex $v\langle a\rangle$ adjacent to an already assigned vertex $u\langle a\rangle$ such that
\begin{itemize}
\item $\sigma_v\in\ZZ/n\ZZ$ if the edge $(u\langle a\rangle,v\langle a\rangle)$ is labelled $t$, and
\item $\sigma_v\in\ZZ/m\ZZ$ if the edge $(u\langle a\rangle,v\langle a\rangle)$ is labelled $t^{-1}$.
\end{itemize}
Then inductively define $x\in BS(m,n)^{(1)}$ to agree with $a^i$ on $B(\langle a\rangle,1)$, and to agree with $a^{c_v}$ on $B(v\langle a\rangle,1)$, where $(u\langle a\rangle$,$v\langle a\rangle)$ is an edge, $l$ is the smallest integer such that $a^l$ fixes the word $v$, and $c_v=c_u+l\sigma_v$. These conditions ensure that $a^{c_v}$ and $a^{c_u}$ agree on the edge $(u\langle a\rangle$,$v\langle a\rangle)$ and hence that $x$ is an automorphism of $\tree_{BS(m,n)}$, which is uniquely identified by the collection $\{\sigma_v:v\in V(\tree_{BS(m,n)})\}$.

\subsection{Automorphism groups of graphs}

Let $\Gamma$ be any graph, and let $\tree$ be the \textit{universal covering tree} $\tree$ of $\Gamma$. There exists a surjection $\psi:\tree\rightarrow\Gamma$ such that the restriction of $\psi$ to $B(v,1)$ is a bijection for all $v\in V(\tree)$.  Then the fundamental group $\pi_1(\Gamma)$ of $\Gamma$ acts naturally on $\tree$ (it is precisely the set of automorphisms $g$ for which $\psi\circ g=\psi$) and there exists $G\leq\Aut(\tree)$ for which 
\[
\pi_1(\Gamma) \hookrightarrow G \overset{\phi}{\twoheadrightarrow} \Aut(\Gamma).
\]
is a short exact sequence, where $\phi(g):\psi(v)\mapsto\psi(g.v)$ defines the group homomorphism induced by the covering map $\psi$ \cite{\Djokovic}. 

It is important to note that whilst $\pi_1(\Gamma)$ is a normal subgroup of $G$, it is generally not normal in $G^{(k)}$.

\begin{proposition}\label{prop:graphs} If $\Gamma$ is a finite graph then $G$ is discrete, and $G^{(k)}=G$ for all $k\geq\mbox{diam}(\Gamma)$.
\begin{proof} Suppose $k\geq\mbox{diam}(\Gamma)$. Then for some $v\in V(\tree)$ we have $\psi(B(v,k))=\Gamma$, and hence if $g\in G$ fixes $B(v,k)$ then $\phi(g)$ is the identity automorphism of $\Gamma$, and hence $g$ is the identity automorphism of $\tree$. Hence $\{1_G\}$ is open in the topology which implies that $G$ is discrete, and hence by Corollary \ref{cor:discrete} $G^{(k)}=G$.
\end{proof}\end{proposition}

\begin{figure}[h]
	\centering
	\includegraphics[height=6cm]{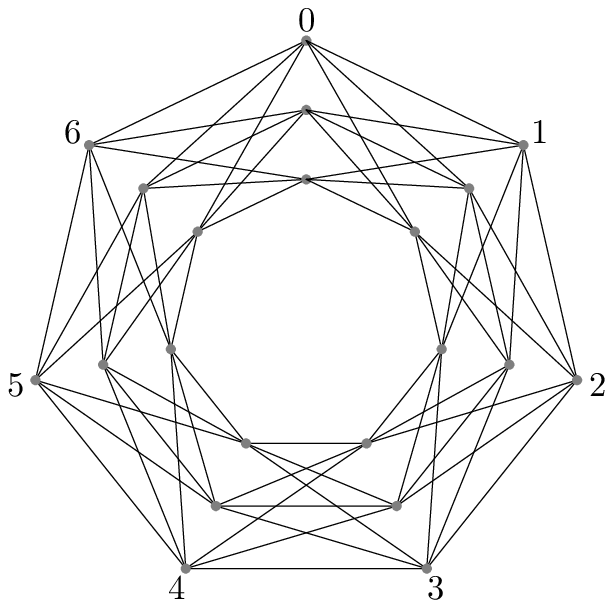}
	\caption{The graph $C(3,7,1)$ referenced in Example \ref{eg:cpr}}
	\label{fig:cpr}
	\end{figure}

\begin{example}\label{eg:cpr}
Let $\Gamma:=C(p,r,1)$ be one of the graphs introduced in \cite{\GardinerPraeger}. Vertices of $\Gamma$ are labelled $\{(i,j):i\in C_r,1\leq j\leq p\}$ with $(i,j)$ adjacent to $(k,l)$ if and only if $k=i\pm 1$ (see Figure \ref{fig:cpr}). Allow also the extension to $r=\infty$, defined to be the infinite graph with vertex set $\{(i,j):i\in\ZZ,1\leq j\leq p\}$ and the same adjacency relation.  Then $\Gamma$ is regular of degree $2p$ and has diameter $\lfloor\frac{r}{2}\rfloor$. The universal cover of $\Gamma$ is $\tree_{2p}$, and there exists a vertex-transitive group $G_{p,r}\leq\Aut(\tree_{2p})$ completing the exact sequence
\[
\pi_1(\Gamma) \hookrightarrow G_{p,r} \overset{\phi}{\twoheadrightarrow} \Aut(\Gamma).
\]
Assume $r\geq 4$, then the local action of $G_{p,r}$ at $v\in V(\tree)$ is isomorphic to $S_p^2\rtimes C_2$, which is independent of $r$.

It turns out that for $r\neq\infty$ the groups $G_{p,r}$ and $G_{p,\infty}$ have the same $k$-closures for all $k<\frac{r}{2}$ (exactly the values for which $\stab_{G_{p,r}}(v)|_{B(v,k)}=\stab_{G_{p,\infty}}(v)|_{B(v,k)}$). We present the argument for $k=1$; the other cases are very similar.

First take $v_0\in V(\tree_{2p})$ to be the base point of the universal cover $\psi:\tree_{2p}\rightarrow C(p,\infty,1)$. Then $G_{p,\infty}$ acts vertex-transitively, and so for any $v\in V(\tree_{2p})$ there exists $h\in G_{p,\infty}$ with $h.v_0=v$, and $\phi(h)\in\Aut(C(p,\infty,1))$. To see that $h\in G_{p,r}$, take the labelling on $\tree_{2p}$ given by $\psi$ and construct a new labelling $(i,j)\mapsto (i\mbox{ mod } r,j)$. We have constructed a universal cover $\psi'$ of $C(p,r,1)$ with base point $v_0$, and a corresponding $\phi':G_{p,r}\twoheadrightarrow\Aut(C(p,r,1)$ such that $h\in G_{p,r}$. Hence $G_{p,r}\cap G_{p,\infty}$ acts vertex-transitively on $\tree_{2p}$. We have already established that both groups have the same local action, and hence by Corollary \ref{cor:kequal} they have the same 1-closure. 

\end{example}

\section{Independence Properties $\Pk{k}$}
\label{sec:PropertyPk}

In this section we define a series of properties, denoted by $\Pk{k}$ for $k\in\NN$, that will be satisfied by the $k$-closure of any group of tree automorphisms. They provide a condition under which the descending series of $k$-closures terminates at $\overline{G}=G^{(k)}$ for some $k\in\NN$. 

\begin{definition}\label{def:Pk}[Property $\Pk{k}$]
Suppose $G\leq\Aut(\tree)$ and fix $k\in\NN$ and an edge $e=(v,w)$. 
Define \[F_{k,e}:=\Fix_G(B(v,k)\cap B(w,k)).\]
Then $G$ satisfies {\em Property $\Pk{k}$} if for any choice of edge $e=(v,w)$,
\[F_{k,e}=\Fix_{F_{k,e}}(\tree_{(v,w)})\Fix_{F_{k,e}}(\tree_{(w,v)}).\]
\end{definition}

\begin{figure}[h]
	\centering
	\includegraphics[width=6cm]{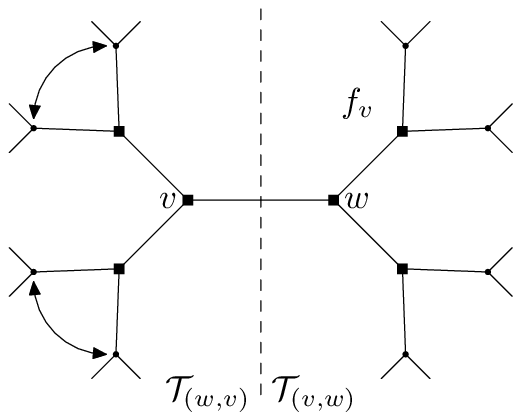}\hspace{0.1cm}
	\includegraphics[width=6cm]{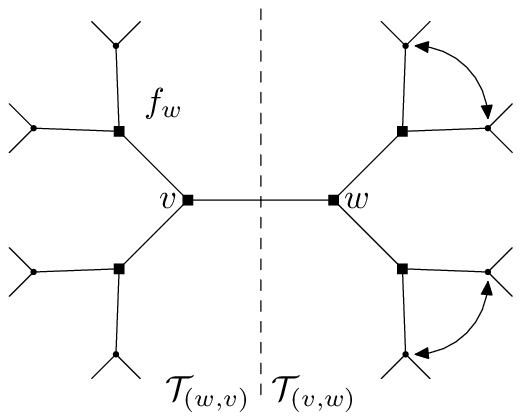}
	\caption{Let $f\in\Fix_G(B(v,2)\cap B(w,2))$ be the automorphism of $\tree_3$ indicated in Figure \ref{fig:non-discrete}, and define $f_v, f_w$ as above. Then $f=f_vf_w$, and if $G$ satisfies Property $\Pk{2}$ then $f_v,f_w\in G$.}
\end{figure}

Note that when $k=1$ then $B(v,1)\cap B(w,1)$ is just the edge $e=(v,w)$ and $F_{1,e}=\Fix_G(e)$. Hence Property $\Pk{1}$ is equivalent to the Independence Property of Amann as discussed in Remark \ref{remark:indep_prop}). Also, Property $\Pk{k}$ is stronger than Property H \cite{\MollerVonk}, which just requires $\Fix_G(\tree_{(v,w)})$ to be non-trivial for every edge in $\tree$.

\begin{proposition}\label{prop:GkhasPk}
Let $G\leq\Aut(\tree)$ and $k\in\NN$. Then $G^{(k)}$ has Property $\Pk{k}$.
\begin{proof} Let $e=(v,w)$ be an edge. Clearly $F_{k,e}\supseteq\Fix_{F_{k,e}}(\tree_{(v,w)})\Fix_{F_{k,e}}(\tree_{(w,v)})$.
 Conversely, suppose $f\in F_{k,e}$. Construct an automorphism $f_1$ by setting $f_1=f$ on all $B(u,k)$ where $u\in \tree_{(w,v)}$, and $f_1$ trivial on all $B(t,k)$ where $t\in\tree_{(v,w)}$. Similarly construct $f_2$ by setting $f_1=f$ on all $B(t,k)$ where $t\in \tree_{(v,w)}$, and $f_2$ trivial on all $B(u,k)$ where $u\in\tree_{(w,v)}$. It is clear that $f=f_1f_2$, and since $f_1$ and $f_2$ agree with either $f$ or the identity on every ball, each is in $G^{(k)}$. Hence $f=f_1f_2\in\Fix_{F_{k,e}}(\tree_{(v,w)})\Fix_{F_{k,e}}(\tree_{(w,v)})$ as required.
\end{proof}\end{proposition}

\begin{proposition}\label{prop:Pk+1} Let $G\leq\Aut(\tree)$ and $k\in\NN$. If $G$ satisfies Property $\Pk{k}$ then it also satisfies Property $\Pk{k+1}$.
\begin{proof}
Let $e=(v,w)$ be any edge, and suppose $x\in F_{k+1,e}$. Since $F_{k+1,e}\leq F_{k,e}$ and $G$ satisfies Property $\Pk{k}$ then $x=x_1x_2$, where $x_1\in\Fix_{F_{k,e}}(\tree_{(v,w)})$ and $x_2\in\Fix_{F_{k,e}}(\tree_{(w,v)})$. Then $x_1x_2$ fixes $\mathcal{B}:=B(v,k+1)\cap B(w,k+1)$, but since $x_2$ fixes $\mathcal{B}\cap\tree_{(w,v)}$ then so must $x_1$. But $x_1$ also fixes $\mathcal{B}\cap\tree_{(v,w)}$ and hence $x_1\in F_{k+1,e}$. Similarly $x_2\in F_{k+1,e}$, and therefore
$F_{k+1,e}=\Fix_{F_{k+1,e}}(\tree_{(v,w)})\Fix_{F_{k+1,e}}(\tree_{(w,v)})$; that is, $G$ satisfies Property $\Pk{k+1}$.
\end{proof}\end{proposition}

As mentioned earlier, taking the $l$-closure of $G^{(k)}$ when $l\leq k$ obtains the $l$-closure of $G$. On the other hand, if $l>k$ the above result implies that $G^{(k)}$ satisfies Property $\Pk{l}$. The following theorem implies that the $l$-closure of $G^{(k)}$ is equal to $G^{(k)}$.

\begin{theorem}\label{thm:PkandGk} 
Let $G\leq\Aut(\tree)$. If $G$ satisfies Property $\Pk{k}$ then $G^{(k)}=\overline{G}$, and if $G^{(k)}=\overline{G}$ then $\overline{G}$ satisfies Property $\Pk{k}$.
\begin{proof}
Suppose $G$ has property $\Pk{k}$; it will be proven by induction that $G^{(r)}=G^{(k)}$ for all $r\geq k$. For the inductive hypothesis, assume for some such $r$ that $G^{(r)}=G^{(k)}$, and take $g\in G^{(k)}$ and $v$ a vertex. It will be shown that there exists $y\in G$ such that $y\mid_{B(v,r+1)}=g\mid_{B(v,r+1)}$ and hence $g$ is in the $(r+1)$-closure of $G$.

By the inductive hypothesis $g\in G^{(r)}$, and hence there exists $z\in G$ so that $z\mid_{B(v,r)}=g\mid_{B(v,r)}$, which also implies $z^{-1}g\in G^{(r)}$ fixes $B(v,r)$.
Let $v_1,\ldots, v_m$ be the vertices that are distance exactly $r-k+1$ from $v$. Then for all $i=1,\ldots,m$ there exists $a_i\in G$ such that $a_i\mid_{B(v_i,k)}=(z^{-1}g)\mid_{B(v_i,k)}$. Each $a_i$ fixes the intersection $B(v_i,k)\cap B(v,r)$, which is equal to $B(v_i,k)\cap B(w_i,k)$ where $w_i$ is the vertex adjacent to $v_i$ that is closest to $v$. Since $G$ has property $\Pk{k}$ then $a_i=b_ic_i$ where $b_i$ fixes $\tree_{(v_i,w_i)}$ and $c_i$ fixes $\tree_{(w_i,v_i)}$.

Now $b_1b_2\cdots b_m$ fixes the ball of radius $r$ around $v$, and in each ball $B(v_i,r)$ it acts like $z^{-1}g$ (since each $b_i$ only acts non-trivially on just one $B(v_i,r)$). Therefore $(zb_1b_2\cdots b_m)^{-1}g\in G$ fixes $B(v,r+1)$. Taking $y=zb_1b_2\cdots b_m$ proves that $g\in G^{(r+1)}$, so by induction $G^{(r)}=G^{(k)}$ for all $r\geq k$. By Proposition \ref{prop:basic} we have $G^{(k)}=\bigcap G^{(r)}=\overline{G}$.

The second assertion follows directly from Proposition \ref{prop:GkhasPk}.
\end{proof}
\end{theorem}

The main result of \cite{\AbertGlasner} implies that there exist free groups acting on regular trees which are dense in either $\Aut(\tree)$ or the simple subgroup $\Aut(\tree)^+$ (see Section \ref{sec:simplicity}). Such a group cannot have Property $\Pk{1}$ but its closure does.

\begin{corollary}\label{cor:PkandGk}
Let $G\leq\Aut(\tree)$ and $k\in\NN$. Then $\overline{G}$ does not satisfy Property $\Pk{k}$ for any $k$ if and only if $G$ has infinitely many distinct $k$-closures.
\begin{proof} 
If $G$ has finitely many $k$-closures there is a smallest one $G^{(j)}$. But then $G^{(j)}=\bigcap_{k\in \NN} G^{(k)}$ which by Proposition \ref{prop:basic}(\ref{prop:basic3}) is equal to $\overline{G}$. Hence by Theorem \ref{thm:PkandGk} $\overline{G}$ has Property $\Pk{j}$. On the other hand if $G$ has infinitely many $k$-closures then the sequence of $G^{(k)}$ is never constant. Hence for all $k$ we have $G^{(k)}\neq\overline{G}$, and so by Theorem \ref{thm:PkandGk} $\overline{G}$ does not satisfy Property $\Pk{k}$.
\end{proof}\end{corollary}

\subsection{Examples}

Recall that in Section \ref{sec:non-discrete-i} we showed that the $k$-closures of the group $G=PSL(2,\mathbb Q_p)$ are all distinct. It follows that the closure of $PSL(2,\mathbb Q_p)$ is not equal to any of its $k$-closures, and hence does not satisfy Property $\Pk{k}$ for any $k$. 

\begin{proposition}\label{prop:BSnoPk} When $m,n$ are coprime, $G=BS(m,n)$ acting on its Bass-Serre tree (as discussed in Section \ref{sec:BS}) does not satisfy Property $\Pk{k}$ for any $k\in\NN$.

\begin{proof} Consider the edge $(v,w):=(\langle a\rangle,t\langle a\rangle)$. We show that for this edge the fixators of the two semitrees $\tree_{(v,w)}$ and $\tree_{(w,v)}$ under $G$ are both trivial. Since by Lemma \ref{lem:BS} $\Fix_G(B(v,k)\cap B(w,k))$ is non-trivial, it then follows that $G$ does not satisfy Property $\Pk{k}$.

Suppose $a^{cn^j}$ ($c>0$ and not divisible by $n$) is a non-trivial automorphism in the fixator of the semitree $T_{(v,w)}$. Such an automorphism must have this form since it fixes the vertex $\langle a \rangle$. The semitree $T_{(v,w)}$ contains vertices $(ta)^it\langle a\rangle$ for all $i\in\NN$. Assume that $i\geq j$, then $a^{cn^j}(ta)^it=(ta)^ja^{cm^j}(ta)^{i-j}t$. To continue passing $a^{cm^j}$ through this word we require $cm^j$ to be a multiple of $n$. Since $c$ is not divisible by $n$ then $m^j$ is a multiple of $n$, which contradicts our assumption that $m,n$ are coprime.

A similar argument holds for the semitree $T_{(tv,v)}$, which contains vertices labelled by $(t^{-1}a)^it^{-1}\langle a\rangle$ for all $i\in\NN$.
\end{proof}
\end{proposition}

\begin{remark}
Recall that for any discrete group $G\leq\Aut(\tree)$ there exists an edge $(v,w)$ and a value of $k$ for which $\Fix_G(B(v,k)\cap B(w,k))$ is trivial. In this case Property $\Pk{k}$ is trivially satisfied and $G^{(k)}=G$.
In our illustrative example from Section \ref{sec:k-closure} (Example \ref{eg:egfirst}) this holds for $k=2$. For the examples obtained from the automorphism group of some finite graph (see Example \ref{eg:cpr}) this holds for $k$ greater than or equal to half the diameter of the graph.
\end{remark}

\section{Property $P_k$.}\label{sec:propertytits}
In this section we define the natural generalisation of Tits' Property $P$, which we call \textit{Property $P_k$}, and relate it to the independence properties defined in the previous section. Property $P_k$ will be used in the next section to prove our simplicity result. The notation defined in the next two definitions will be used throughout both sections.

\begin{definition} If $X$ is a subtree of $\tree$ then let $X^k$ denote the subtree of $\tree$ spanned by the set $\{x\in V(\tree):d(x,X)\leq k\}$ of vertices at distance at most $k$ from $X$.
\end{definition}

\begin{definition}\label{def:titsPk} Suppose $G\leq\Aut(\tree)$, let $C$ be any path (finite or infinite) in $\tree$ and define $\pi$ to be the projection of $V(\tree)$ onto $V(C)$. For each $x\in V(C)$ define $F_x$ to be the permutation group acting on $\pi^{-1}(x)$ induced by $\Fix_G(C^{k-1})$. Then we say $G$ satisfies \textit{Property $P_k$} if and only if for all such $C$ the natural homomorphism $\Phi:\Fix_G(C^{k-1})\rightarrow\prod_{x\in V(C)}F_x$ is an isomorphism.
\end{definition}

It is immediate from this definition that if a group satisfies Property $P_k$ then it must also satisfy Property $\Pk{k}$.

\begin{proposition}\label{prop:chainsPk} Suppose $G\leq\Aut(\tree)$ satisfies Property $\Pk{k}$, and let $C$ be any finite path in $\tree$. Let $\pi$ denote the projection of $V(\tree)$ onto $V(C)$, and for each $x\in V(C)$ define $F_x$ to be the permutation group acting on $\pi^{-1}(x)$ induced by $\Fix_G(C^{k-1})$. Then $\Phi:\Fix_G(C^{k-1})\rightarrow\prod_{x\in V(C)}F_x$ is an isomorphism.
\end{proposition}

\begin{proof}
Let $C$ be a path of length $N$ defined by vertices $x_0,...,x_N$. We show by induction on $N$ that if $G$ satisfies Property $\Pk{k}$ then $\Phi$ is bijective. When $N=1$ the path $C$ is a single edge $(x_0,x_1)$ and Property $\Pk{k}$ implies that $\Phi$ is an isomorphism.

For the inductive hypothesis let $N>1$ and assume $\Phi$ is bijective for all paths of length $N-1$. It is clear that $\Phi$ is injective. Now take any $\prod_{i=1}^Nf_i$, $f_i\in F_{x_i}$. Let $C'$ denote the path $x_0...x_{N-1}$ of length $N-1$, let $\pi'$ denote the projection of $V(\tree)$ onto $V(C')$, define ${F'}_x$ to be the permutation group acting on $\pi'^{-1}(x)$ induced by $\Fix_G(C'^k)$ and let $\Phi'$ denote the natural homomorphism.

Since $\pi'^{-1}(x_{N-1})$ is the disjoint union of $\pi^{-1}(x_{N-1})$ and $\pi^{-1}(x_N)$ define $f'_{N-1}\in{F'}_{x_{N-1}}$ to agree with $f_{N-1}$ on $\pi^{-1}(x_{N-1})$ and with $f_N$ on $\pi^{-1}(x_N)$. By inductive hypothesis there exists $f\in\Fix_G(C'^{k-1})$ such that $\Phi'(f)=f_1...f_{N-2}f'_{N-1}$. But since $f$ agrees with $f_N$ on $\pi^{-1}(x_N)$ then it fixes $\pi^{-1}(x_N)\cap C^{k-1}$. Hence $f\in\Fix_G(C^{k-1})$ and $\Phi(f)=\prod_{i=1}^Nf_i$ as required.
\end{proof}

\begin{corollary}\label{cor:chainsPk}
Let $G$ be a closed subgroup of $\Aut(\tree)$. Then $G$ satisfies Property $\Pk{k}$ if and only if $G$ satisfies Property $P_k$.
\end{corollary}
\begin{proof}
From Proposition \ref{prop:chainsPk} it remains to show that if $G$ satisfies Property $\Pk{k}$, and $C$ is either an infinite or doubly-infinite path, then $\Phi$ is an isomorphism. We give the proof for the second case only, as this is the case we will encounter in the next section and both proofs are essentially the same.

Index the vertices $x_i$ in $C$ and consider an element $\prod f_i$ of $\prod_{i\in \ZZ}F_{x_i}$. Define $C_n$ to be the path from $x_{-n}$ to $x_n$. Apply Proposition \ref{prop:chainsPk} to $C_n$, in each case denoting the projection of $V(\tree)$ onto $V(C_n)$ by $\pi_n$ and the natural isomorphism by $\Phi_n$. This results in a sequence $\{g_n\}_{n\in\NN}$ where $g_n\in\Fix_G(C_n^{k-1})$, and $\Phi_n(g_n)=f'_{-n}f_{-n+1}...f_{n-1}f'_n$, where $f'_n$ is the permutation on $\pi_n^{-1}(x_n)$ that, for $j\geq n$, agrees with each $f_j$ on $\pi^{-1}(x_j)$ (and similarly for $f'_{-n}$). Since $G$ is closed the sequence $\{g_n\}$ converges to some $g\in\Fix_G(C^{k-1})$. But $\Phi_n(g_n)$ and $\Phi(g)$ agree on the sets $\pi^{-1}(x)$ for all $x\in V(C_{n-1})$. Hence $\Phi(g)$ is equal to the limit of the sequence $\{\Phi_n(g_n)\}$, which is $\prod f_i$.
\end{proof}

\section{Simplicity}
\label{sec:simplicity}

In this section we will prove our main simplicity result (Theorem \ref{thm:simple}), which utilises Property $\rm{P}_k$. The proof of this theorem follows the same process as the proof of \cite[Theoreme 4.5]{\Tits}. The following lemma, which is our analogue of \cite[Lemme 4.3]{\Tits}, is required.

\begin{lemma}\label{prop:commPk} Suppose $G\leq\Aut(\tree)$ is a closed subgroup and fix $k\in\NN$. Let $h\in G$ induce on some doubly-infinite path $C$ in $\tree$ a non-trivial translation. Let $K$ denote the fixator of $C^{k-1}$ in $G$. If $G$ satisfies Property $P_k$ then \[K=[h,K]:=\{hgh^{-1}g^{-1}:g\in K\}.\]
\end{lemma}
\begin{proof} Clearly $h$ stabilises $C^k$, and hence $hgh^{-1}\in K$. This implies that $K\supseteq[h,K]$. Now suppose that $f\in K$; we must find $g\in K$ such that $hgh^{-1}g^{-1}=f$. Let $a$ be the amplitude of the translation $h$ and form a natural bijection between $V(C)$ and $\ZZ$. By Definition \ref{def:titsPk} $f=\prod f_z$ where $f_z\in F_z$, and $g$ can be defined by finding appropriate $g_z\in F_z$ for all $z\in \ZZ$. For each $z\in\ZZ$ notice that $h$ induces an isomorphism $\eta_z:F_z\rightarrow F_{z+a}$ defined by $\eta_z(x)=hxh^{-1}$. We define $g_z$ inductively as follows: if $0\leq z\leq a-1$ then $g_z$ is arbitrarily chosen in $F_z$; if $z\geq a$ then $g_z=f_z^{-1}.\eta_{z-a}(g_{z-a})$; if $z<0$ then $g_z=\eta_z^{-1}(f_{z+a}g_{z+a})$. It is easy to check that $g=\prod g_z$ satisfies $hgh^{-1}g^{-1}=f$ as required.
\end{proof}

We define the following subgroup of a group of tree automorphisms, which will be shown to be simple in the theorem below.
\begin{definition} Let $G\leq\Aut(\tree)$ be closed and fix $k\in\NN$. Then $G^{+_k}$ is the subgroup of $G$ generated by all elements $g\in G$ for which there exists $(v,w)\in E(\tree)$ such that $g$ fixes $B(v,k)\cap B(w,k)$ (which is equivalent to $g$ fixing $B(v,k-1)\cup B(w,k-1)$). 
\end{definition}

Recall that if $G$ is discrete then it can only satisfy Property $\rm{P}_k$ trivially for any $k\in\NN$; in this case $G^{+_k}$ is trivial. On the other hand it is immediate from the topology on $G<\Aut(\tree)$ that if $G$ is non-discrete then $G^{+_k}$ is nontrivial.

\begin{theorem}\label{thm:simple}
Let $\tree$ be a tree and fix $k\in\NN$. Suppose $G\leq\Aut(\tree)$ does not stabilise a proper non-empty subtree or an end of $\tree$, and satisfies Property $P_k$. Then every nontrivial subgroup of $G$ normalised by $G^{+_k}$ contains $G^{+_k}$; in particular $G^{+_k}$ is simple (or trivial).
\begin{proof}
Assume $G^{+_k}$ is non-trivial, let $H$ be a nontrivial subgroup of $G$ normalised by $G^{+_k}$, and $e=(v,w)$ be any edge of $\tree$. To prove the theorem it suffices to show that $H$ contains the fixator of the semi-tree $\tree_{(v,w)}$ in $F_{k,e}:=\Fix_G(B(v,k)\cap B(w,k))$ ($H$ will contain the fixator of $\tree_{(w,v)}$ in $F_{k,e}$ by a similar argument). Then it follows from Property $\Pk{k}$ that $H$ contains $F_{k,e}=\Fix_{F_{k,e}}(\tree_{(v,w)})\Fix_{F_{k,e}}(\tree_{(w,v)})$ for any edge $e$, and hence $H$ contains all generators of $G^{+_k}$.

Since $G^{+_k}$ is normal in $G$, by Lemma \ref{lem:normalised} $G^{+_k}$ does not stabilise a proper non-empty subtree or an end of $\tree$. We assumed that $H$ is normalised by $G^{+_k}$, so again by Lemma \ref{lem:normalised} $H$ does not stabilise a proper non-empty subtree or an end of $\tree$. Therefore by Proposition \ref{prop:T3.4} there exists a doubly-infinite path $C$ of $\tree$ and a non-trivial translation $h\in H$ on $C$. We will show that it may be assumed that $C\subset\tree_{(v,w)}$.

For any vertex $v$ of $C$ the orbit $H.v$ has non-empty intersection with $\tree_{(v,w)}$ by Lemma \ref{lem:T4.1}; that is, there exists $g\in H$ with $g(C)\cap\tree_{(v,w)}\neq\emptyset$. By replacing $C,h$ by $g(C),ghg^{-1}$ we can assume that $C\cap\tree_{(v,w)}$ is non-empty. In particular this means $C\cap\tree_{(v,w)}$ is at least an infinite path. Let $b,b'$ be the ends of $\tree$ for which $C$ contains representatives, at least one of which, say $b$, has representatives in $C\cap\tree_{(v,w)}$. Since $H$ does not stabilise a proper non-empty subtree or an end of $\tree$, there exists some $l\in H$ such that $l(b')\notin\{b,b'\}$. In addition, $l^{-1}(C)$ is a doubly-infinite path that does not contain any representative of $b'$ (since $l(b')\notin\{b,b'\}$).

Now let $\pi$ represent the projection of $\tree$ onto $C$; that is, for all $x\in V(\tree)$ define $\pi(x)$ to be the vertex of $C$ closest to $x$. There are two possibilities for the image $\pi(\tree_{(w,v)})$ of the other semi-tree; either $\pi(\tree_{(w,v)})$ is a single vertex in $\tree_{(v,w)}$ (when $C$ is contained in $\tree_{(v,w)}$), or $C':=\pi(\tree_{(w,v)})$ is the representative of $b'$, an end of $\tree_{(w,v)}$, that begins at $w$. Now since $l^{-1}(C)$ does not contain any representative of $b'$ the image $\pi(l^{-1}(C))$ must be contained in some representative of $b$, which is also contained in $C$. Define $C''$ to be the shortest such representative of $b$ (i.e. $\pi(l^{-1}(C))\subseteq C''\subseteq C$). Choose an integer $n$ such that $h^n(C'')$ and $C'$ are as far away as we like, say distance $k$. Given such $n$ the chain $h^n(l^{-1}(C))$ is disjoint from $\tree_{(w,v)}$ and hence contained in $\tree_{(v,w)}$. Replacing $C,h$ by $h^n(l^{-1}(C)),h^nl^{-1}hlh^{-n}$ we can assume that $C\subset\tree_{(v,w)}$, as is $C^{k-1}:=span\{x\in V(\tree):d(x,C)\leq k\}$ (by our choice of $n$).

Let $K$ denote the fixator of $C^{k-1}$ in $F_{k,e}$; clearly $K\subset G^{+_k}$. Lemma \ref{prop:commPk} implies that $K=[h,K]$, which is in $H$ (since $H$ is normalised by $G^{+_k}$). Also since $C^k\subset\tree_{(v,w)}$ then $K$ contains the fixator of $\tree_{(v,w)}$ in $F_{k,e}$. Therefore $\Fix_{F_{k,e}}(\tree_{(v,w)})\subset H$ as required.
\end{proof}
\end{theorem}

\section{Constructing Simple Groups}\label{sec:other}

The preceding results underpin a general method for constructing simple groups acting on trees that, beginning with some arbitrary group $G$, constructs its $k$-closure (which has Property $\Pk{k}$) and applies Theorem \ref{thm:simple} to obtain the simple group $(G^{(k)})^{+_k}$. In this section further properties of $(G^{(k)})^{+_k}$ are proven.

\begin{lemma}\label{lem:simplegroups}
Suppose $G\leq\Aut(\tree)$ does not stabilise any proper subtree of $\tree$. Then:
\begin{enumerate}
\item $(G^{(k)})^{+_k}$ is an open (hence closed) subgroup of $G^{(k)}$;
\item $(G^{(k)})^{+_k}$ is non-discrete if and only if $G^{(k)}$ is non-discrete; and
\item $(G^{(k)})^{+_k}$ satisfies Property $\Pk{k}$.
\end{enumerate}
\begin{proof}
(i) This follows since the group is generated by the open sets $\Fix_{G^{(k)}}(B(v,k)\cap B(w,k))$.

(ii) If $G^{(k)}$ is discrete then all of its subgroups are discrete. Similarly if $G^{(k)}$ is non-discrete then all of its open subgroups are non-discrete.


(iii) Consider an edge $e=(v,w)$ and let $F_{k,e}:=\Fix_{G^{(k)}}(B(v,k)\cap B(w,k))$. By Proposition \ref{prop:GkhasPk} $G^{(k)}$ has Property $\Pk{k}$, which implies that $F_{k,e}$ is equal to the product $\Fix_{F_{k,e}}(\tree_{(v,w)})\Fix_{F_{k,e}}(\tree_{(w,v)})$. By definition $F_{k,e}$ and the factors belong to $(G^{(k)})^{+_k}$, and are therefore equal to their counterparts in $(G^{(k)})^{+_k}$. Hence $(G^{(k)})^{+_k}$ also satisfies Property $\Pk{k}$.
\end{proof}
\end{lemma}

\begin{theorem}\label{thm:simplegroups} Suppose $G\leq\Aut(\tree)$ does not stabilise any proper subtree of $\tree$. Then $(G^{(r)})^{+_r}\leq (G^{(k)})^{+_k}$ for all $r>k$, with equality if and only if $G^{(r)}=G^{(k)}$.

\begin{proof} For the first part, note that $G^{(r)}\leq G^{(k)}$, and hence every generator in $(G^{(r)})^{+_r}$ is a generator in $(G^{(k)})^{+_k}$.

Suppose $G^{(r)}=G^{(k)}$ and let $g\in(G^{(k)})^{+_k}$ be a generator. Then there exists an edge $(v,w)$ such that $g$ fixes $B(v,k)\cap B(w,k)$. By Lemma \ref{lem:simplegroups}(iii) $(G^{(k)})^{+_k}$ has Property $\Pk{k}$ and hence $g=g_1g_2$, where $g_1\in\Fix_{F_{k,e}}(\tree_{(v,w)})$ and $g_2\in\Fix_{F_{k,e}}(\tree_{(w,v)})$. Since $g_1\in G^{(k)}$ (and hence $G^{(r)}$) and $g_1$ fixes $\tree_{(v,w)}$, then there exists some edge $(t,u)\in\tree_{(v,w)}$ such that $g$ fixes $B(t,r)\cap B(u,r)$. Hence $g_1$ is a generator in $(G^{(r)})^{+_r}$. Similarly $g_2$ is a generator in $(G^{(r)})^{+_r}$, and hence $g\in (G^{(r)})^{+_r}$. 

On the other hand assume $(G^{(r)})^{+_r}=(G^{(k)})^{+_k}$. Let $x\in G^{(k)}$, $v\in V(\tree)$ and pick some $g\in G$ such that $x|_{B(v,k)}=g|_{B(v,k)}$. Then $g^{-1}x$ fixes $B(v,k)$ and hence is in $(G^{(k)})^{+_k}$. Then from the assumption $g^{-1}x\in G^{(r)}$, and since $g\in G^{(r)}$ then $gg^{-1}x=x\in G^{(r)}$.
\end{proof}\end{theorem}

\begin{corollary}\label{cor:simplegroups} Suppose that $G\leq\Aut(\tree)$ does not stabilise any proper subtree of $\tree$, and does not satisfy Property $\Pk{k}$ for any $k$. Then there are infinitely many distinct non-discrete simple groups $(G^{(k)})^{+_k}$.
\begin{proof}
By Corollary \ref{cor:PkandGk} there are infinitely many distinct $G^{(k)}$, which implies that $G^{(k)}\neq \overline{G}$ for all $k$. By Theorem \ref{thm:simplegroups} there are infinitely many distinct simple groups $(G^{(k)})^{+_k}$. By Corollary \ref{cor:discrete} every $G^{(k)}$ is non-discrete, and hence by Lemma \ref{lem:simplegroups} each $(G^{(k)})^{+_k}$ is non-discrete.
\end{proof}\end{corollary}

\begin{example}\label{sec:simpleexamples}
Recall from Section \ref{sec:non-discrete-i} that the $k$-closures of $PSL(2,\mathbb{Q}_p)$ are all non-discrete and distinct. Also recall from Proposition \ref{prop:BSnoPk} that the Baumslag-Solitar group $BS(m,n)$ acting on its Bass-Serre tree does not satisfy Property $\Pk{k}$ for any $k\in\NN$, assuming that $m,n$ are relatively prime. Hence in these examples there are infinitely many distinct (as subgroups of $\Aut(\tree_{m+n})$) non-discrete simple groups $(G^{(k)})^{+_k}$ found by our construction.

We can describe the generators of the simple group $(BS(m,n)^{(1)})^{+_1}$ using the structure of automorphisms in the 1-closure. In Section \ref{sec:BS} we constructed all automorphisms of the 1-closure that fix the vertex labelled by $\langle a\rangle$. Such an automorphism fixes an edge incident on $\langle a\rangle$ (and hence is a generator) if and only if the permutation $c$ assigned to $\langle a\rangle$ is a multiple of either $m$ or $n$. Note that since the group is vertex-transitive, any other generator is the conjugate of one of these automorphisms by a translation, which we can assume to be contained in $BS(m,n)$. 
\end{example}

In order to explicitly describe (and possibly classify) simple groups arising from this process, there are two issues which still need to be addressed. The first problem is obtaining an algebraic description of a group's $k$-closure. This has been determined for the class of universal groups in \cite{\BurgerMozes}, which satisfy Property $P_1$. However we expect this will become increasingly difficult as $k$ increases. Secondly, whilst we can show that two simple groups are distinct as subgroups of $\Aut(\tree)$, this does not ensure they are non-isomorphic as topological groups. We require a method of determining when two groups have isomorphic $k$-closures, and hence contain isomorphic simple subgroups. 

\bibliography{refs} 

\begin{thebibliography}{10}

\bibitem{MR2491892}
Mikl{\'o}s Ab{\'e}rt and Yair Glasner.
\newblock Generic groups acting on regular trees.
\newblock {\em Trans. Amer. Math. Soc.}, 361(7):3597--3610, 2009.

\bibitem{Amann}
Olivier~Eric Amann.
\newblock {\em Groups of Tree-Automorphisms and their Unitary Representations}.
\newblock PhD thesis, Swiss Federal Institute of Technology, 2003.

\bibitem{ChrisPhD}
Christopher Banks.
\newblock Ph{D} thesis.
\newblock In preparation.

\bibitem{MR2956246}
Udo Baumgartner, R{\"o}gnvaldur~G. M{\"o}ller, and George~A. Willis.
\newblock Hyperbolic groups have flat-rank at most 1.
\newblock {\em Israel J. Math.}, 190:365--388, 2012.

\bibitem{MR2356316}
Udo Baumgartner, Bertrand R{\'e}my, and George~A. Willis.
\newblock Flat rank of automorphism groups of buildings.
\newblock {\em Transform. Groups}, 12(3):413--436, 2007.

\bibitem{MR1839488}
Marc Burger and Shahar Mozes.
\newblock Groups acting on trees: from local to global structure.
\newblock {\em Inst. Hautes \'Etudes Sci. Publ. Math.}, 92:113--150, 2000.

\bibitem{MR2739075}
Pierre-Emmanuel Caprace and Nicolas Monod.
\newblock Decomposing locally compact groups into simple pieces.
\newblock {\em Math. Proc. Cambridge Philos. Soc.}, 150(1):97--128, 2011.

\bibitem{MR2217912}
Pierre-Emmanuel Caprace and Bertrand R{\'e}my.
\newblock Simplicit\'e abstraite des groupes de {K}ac-{M}oody non affines.
\newblock {\em C. R. Math. Acad. Sci. Paris}, 342(8):539--544, 2006.

\bibitem{MR2017720}
Lisa Carbone and Howard Garland.
\newblock Existence of lattices in {K}ac-{M}oody groups over finite fields.
\newblock {\em Commun. Contemp. Math.}, 5(5):813--867, 2003.

\bibitem{MR1007714}
Marston Conder and Peter Lorimer.
\newblock Automorphism groups of symmetric graphs of valency {$3$}.
\newblock {\em J. Combin. Theory Ser. B}, 47(1):60--72, 1989.

\bibitem{MR0349486}
D.~{\v{Z}}. Djokovi{\'c}.
\newblock Automorphisms of graphs and coverings.
\newblock {\em J. Combinatorial Theory Ser. B}, 16:243--247, 1974.

\bibitem{MR586434}
Dragomir~{\v{Z}}. Djokovi{\'c} and Gary~L. Miller.
\newblock Regular groups of automorphisms of cubic graphs.
\newblock {\em J. Combin. Theory Ser. B}, 29(2):195--230, 1980.

\bibitem{ElderWillis}
Murray Elder and George~A. Willis.
\newblock Totally disconnected groups from {B}aumslag-{S}olitar groups.
\newblock ArXiv: 1301.4775.

\bibitem{MR1279076}
A.~Gardiner and Cheryl~E. Praeger.
\newblock A characterization of certain families of {$4$}-valent symmetric
  graphs.
\newblock {\em European J. Combin.}, 15(4):383--397, 1994.

\bibitem{MR1668359}
Fr{\'e}d{\'e}ric Haglund and Fr{\'e}d{\'e}ric Paulin.
\newblock Simplicit\'e de groupes d'automorphismes d'espaces \`a courbure
  n\'egative.
\newblock In {\em The {E}pstein birthday schrift}, volume~1 of {\em Geom.
  Topol. Monogr.}, pages 181--248 (electronic). Geom. Topol. Publ., Coventry,
  1998.

\bibitem{MR1703086}
Christophe Kapoudjian.
\newblock Simplicity of {N}eretin's group of spheromorphisms.
\newblock {\em Ann. Inst. Fourier (Grenoble)}, 49(4):1225--1240, 1999.

\bibitem{MR2997026}
R{\"o}gnvaldur~G. M{\"o}ller and Jan Vonk.
\newblock Normal subgroups of groups acting on trees and automorphism groups of
  graphs.
\newblock {\em J. Group Theory}, 15(6):831--850, 2012.

\bibitem{MR1954121}
Jean-Pierre Serre.
\newblock {\em Trees}.
\newblock Springer Monographs in Mathematics. Springer-Verlag, Berlin, 2003.
\newblock Translated from the French original by John Stillwell, Corrected 2nd
  printing of the 1980 English translation.

\bibitem{MR2823790}
Zoran {\v{S}}uni{\'c}.
\newblock Pattern closure of groups of tree automorphisms.
\newblock {\em Bull. Math. Sci.}, 1(1):115--127, 2011.

\bibitem{MR0299534}
Jacques Tits.
\newblock Sur le groupe des automorphismes d'un arbre.
\newblock In {\em Essays on topology and related topics ({M}\'emoires
  d\'edi\'es \`a {G}eorges de {R}ham)}, pages 188--211. Springer, New York,
  1970.

\bibitem{MR1556954}
D.~Van~Dantzig.
\newblock Zur topologischen {A}lgebra. {III}. {B}rouwersche und {C}antorsche
  {G}ruppen.
\newblock {\em Compositio Math.}, 3:408--426, 1936.

\bibitem{MR1299067}
G.~Willis.
\newblock The structure of totally disconnected, locally compact groups.
\newblock {\em Math. Ann.}, 300(2):341--363, 1994.

\bibitem{WiNeretin}
George~A. Willis.
\newblock The scale and flat-rank in {N}eretin's group of spheromorphisms.
\newblock In preparation.

\bibitem{MR2052362}
George~A. Willis.
\newblock Tidy subgroups for commuting automorphisms of totally disconnected
  groups: an analogue of simultaneous triangularisation of matrices.
\newblock {\em New York J. Math.}, 10:1--35 (electronic), 2004.

\bibitem{MR2320465}
George~A. Willis.
\newblock Compact open subgroups in simple totally disconnected groups.
\newblock {\em J. Algebra}, 312(1):405--417, 2007.

\bibitem{MR1691054}
John~S. Wilson.
\newblock {\em Profinite groups}, volume~19 of {\em London Mathematical Society
  Monographs. New Series}.
\newblock The Clarendon Press Oxford University Press, New York, 1998.

\end{thebibliography}
\bibliographystyle{plain}
\end{document}